\documentclass[11pt,fullpage,doublespace]{amsart}

\usepackage{amsfonts, amsmath, amssymb}
\usepackage[all]{xy}
\usepackage{bm}

\newtheorem{theorem}{Theorem}[section]

\newtheorem{lemma}{Lemma}[section]

\newcommand{\mO}{\mathcal{O}}

\newcommand{\ia}{\mathfrak{a}}

\newcommand{\iif}{\mathfrak{f}}
\newcommand{\im}{\mathfrak{m}}
\newcommand{\ip}{\mathfrak{p}}
\newcommand{\iq}{\mathfrak{q}}

\newcommand{\Z}{\mathbb{Z}}

\newcommand{\Q}{\mathbb{Q}}

\newcommand{\F}{\mathbb{F}}

\theoremstyle{remark}

\title{On the Average Exponent of CM elliptic curves modulo $p$ }
\author{Kim, Sungjin }

\begin{document}
    \maketitle

    \begin{abstract}
    Let $E$ be an elliptic curve defined over $\Q$ and with complex multiplication by $\mO_K$, the ring of integers in an imaginary quadratic field $K$. It is known that $E(\F_p)$ has a structure
    \begin{equation}
    E(\F_p)\simeq \Z/d_p\Z \oplus \Z/e_p\Z.
    \end{equation}
    with $d_p|e_p$.
    We give an asymptotic formula for the average order of $e_p$, with improved error term, and upper bound estimate for the average of $d_p$. 
    \end{abstract}

    \section{Introduction}
    Let $E$ be an elliptic curve over $\Q$, and $p$ be a prime of good reduction. Denote $E(\F_p)$ the group of $\F_p$-rational points of $E$. It is known that $E(\F_p)$ has a structure
    \begin{equation}
    E(\F_p)\simeq \Z/d_p\Z \oplus \Z/e_p\Z.
    \end{equation}
    with $d_p|e_p$.
    By Weil's bound, we have
    \begin{equation}
    |E(\F_p)|=p+1-a_p
    \end{equation}
    with $|a_p|<2\sqrt{p}$.
    We fix some notations before stating results. Let $E[k]$ be the $k$-torsion points of the group $E(\overline{\Q})$. Denote $\Q(E[k])$ the $k$-th division field, which is obtained by adjoining coordinates of $E[k]$. Denote $n_k$ the field extension degree $[\Q(E[k]):\Q]$.
    Recently, T.Freiberg and P.Kurlberg ~\cite{TP} started investigating the average order of $e_p$(In the summation, we take $0$ in place of $e_p$ when $E$ has a bad reduction at $p$). They obtained that there exists a constant $c_E\in (0,1)$ such that
    \begin{equation}
    \sum_{p\leq x} e_p = c_E \textrm{Li}(x^2)+O(x^{19/10}(\log x) ^{6/5})
    \end{equation}
    under GRH, and
    \begin{equation}
    \sum_{p\leq x} e_p = c_E \textrm{Li}(x^2)+O(x^2\log \log \log x / \log x \log \log x).
    \end{equation}
    unconditionally when $E$ has CM.
    More recently, J.Wu ~\cite{JW} improved their error terms in both cases
    \begin{equation}
    \sum_{p\leq x} e_p = c_E \textrm{Li}(x^2)+O(x^{11/6}(\log x) ^{1/3})
    \end{equation}
    under GRH, and
    \begin{equation}
    \sum_{p\leq x} e_p = c_E \textrm{Li}(x^2)+O(x^2/(\log x)^{9/8}).
    \end{equation}
    unconditionally when $E$ has CM.

    In this paper we improve the unconditional error term in CM case by using a number field analogue of Bombieri-Vinogradov theorem due to ~\cite[Theorem 1]{H}.

    \begin{theorem}
    Let $E$ be a CM elliptic curve defined over $\Q$ and with complex multiplication by $\mO_K$, the ring of integers in an imaginary quadratic field $K$. Let $N$ be the conductor of $E$. Let $A,B>0$, and $N\leq (\log x)^A$. Then we have
    $$
    \sum_{p\leq x, p\nmid N} e_p = c_E \textrm{Li}(x^2)+O_{A,B}(x^2/(\log x)^B).
    $$
    where
    $$
    c_E = \sum_{k=1}^{\infty}\frac{1}{n_k}\sum_{dm=k}\frac{\mu(d)}{m}.
    $$
    \end{theorem}
    We are also interested in the average behavior of $d_p$. For the average of $d_p$, we have an upper bound result. We apply the number field analogue of Brun-Titchmarsh inequality due to ~\cite[Theorem 4]{HL}.
    \begin{theorem}
    Let $E$ be a CM elliptic curve defined over $\Q$ and with complex multiplication by $\mO_K$, the ring of integers in an imaginary quadratic field $K$. Let $N$ be the conductor of $E$. Let $A>0$, and $N\leq (\log x)^A$. Then we have
    $$
    \sum_{p\leq x, p\nmid N} d_p \ll_A x\log\log x,
    $$
    where the implied constant is absolute.
    \end{theorem}
    Note that the upper bound is sharper than the trivial bound $\ll x\log x$.

    \section{Preliminaries}
    \begin{lemma}
    Let $E$ be a CM elliptic curve defined over $\Q$ and with complex multiplication by $\mO_K$. Then for $k>2$,
    $$
    \phi(k)^2 \ll n_k \ll k^2
    $$
    where $\phi$ is the Euler function.
    \end{lemma}

    \begin{lemma}
    Let $E$ be an elliptic curve over $\Q$, and $p$ be a prime of good reduction. Then
    $$k|d_p \Leftrightarrow p \textrm{ splits completely in }\Q(E[k]).$$
    \end{lemma}
    \begin{proof}
    See ~\cite[page 159]{M}.
    \end{proof}

    Let $N$ be the conductor of $E$, and denote
    $$\pi_E(x;k)=\# \{p\leq x: p \nmid N, \ \ p \textrm{ splits completely in $\Q(E[k])$}\}$$

    \begin{lemma}
    For $2\leq k\leq 2\sqrt{x}$, we have
    $$
    \pi_E(x;k)\ll \frac{x}{k^2}
    $$
    where the implied constant is absolute.
    \end{lemma}
    \begin{proof}
    See A.Cojocaru ~\cite[Lemma 2.6]{AC}, and note that there are only nine possibilities of $K$.
    \end{proof}

    We state some class field theory background. For the proofs, see ~\cite[Lemma 2.6, 2.7]{AM}.
    \begin{lemma}
    If $k\geq 3$ then $\Q(E[k])=K(E[k])$.
    \end{lemma}
    \begin{lemma}
    Let $E/\Q$ have CM by $\mO_K$ and $k\geq 1$ be an integer. Then there is an ideal $\iif$ of $\mO_K$ and $t(k)$ ideal classes mod $k\iif$ with the following property:

    If $\ip$ is a prime ideal of $\mO_K$ with $\ip \nmid k\iif$, then
    $$\ip \textrm{ splits completely in }K(E[k])\Leftrightarrow \ip\sim\im_1, \textrm{ or } \im_2, \textrm{ or } \cdots, \textrm{ or } \im_{t(k)} \textrm{ mod }k\iif.$$
    Moreover
    $$t(k)[K(E[k]):K]=h(k\iif),$$
    where
    $$t(k)\leq c\phi(\iif)\prod_{\ip\mid\iif}\left(1+\frac{1}{N(\ip)-1}\right).$$
    Here $c$ is an absolute constant and $\phi(\iif)$ is the number field analogue of the Euler function.
    \end{lemma}

    Let $\pi_K(x;\iq,\ia)=\# \{\ip: \textrm{ prime ideal; } N(\ip)\leq x, \textrm{ and }\ip\sim\ia\textrm{ mod }\iq\}$.
    The following is a number field analogue of the Bombieri-Vinogradov theorem due to Huxley ~\cite[Theorem 1]{H}.
    \begin{lemma}
    For each positive constant $B$, there is a positive constant $C=C(B)$ such that
    $$\sum_{N(\iq)\leq Q} \max_{(\ia,\iq)=1}\max_{y\leq x}\frac{1}{T(\iq)}\left|\pi_K(y;\iq,\ia)-\frac{\textrm{Li}(y)}{h(\iq)}\right|\ll\frac{x}{(\log x)^B},$$
    where $Q=x^{1/2}(\log x)^{-C}$. The implied constant depends only on $B$ and on the field $K$.
    \end{lemma}
    There is a number field analogue of Brun-Titchmarsh inequality due to J. Hinz and M. Lodemann ~\cite[Theorem 4]{HL}.
    \begin{lemma}
    Let $\mathfrak{H}$ denote any of the $h(\iq)$ elements of the group of ideal-classes mod $\iq$ in the narrow sense. If $1\leq N\iq <X$, then
    $$
    \sum_{\substack{{N\ip <X}\\{\ip\in\mathfrak{H}}}}1\leq 2\frac{X}{h(\iq)\log\frac{X}{N\iq}}\left\{1+O\left(\frac{\log\log 3\frac{X}{N\iq}}{\log\frac{X}{N\iq}}\right)\right\}.
    $$
    \end{lemma}

    We are now ready to prove Theorem 1.1. From now on, $E$ is an elliptic curve over $\Q$ that has CM by $\mO_K$, where $K$ is one of the nine imaginary quadratic field with class number 1. Let $N$ be the conductor of $E$.
    \section{Proof of the theorem 1.1}
    By Weil's bound, we have
    \begin{equation}
    \sum_{p\leq x, p\nmid N} e_p = \sum_{p\leq x, p\nmid N} \frac{p}{d_p}+O\left(\frac{x^{3/2}}{\log x}\right).
    \end{equation}
    As shown in both ~\cite{TP} and  ~\cite{JW}, we use the following elementary identity
    \begin{equation}
    \frac{1}{k}=\sum_{dm \mid k}\frac{\mu(d)}{m}.
    \end{equation}
    Thus we obtain
    \begin{align*}
    \sum_{p\leq x, p\nmid N}\frac{p}{d_p} &=\sum_{p\leq x, p\nmid N} p\sum_{dm|d_p}\frac{\mu(d)}{m}\\
     &=\sum_{k\leq 2\sqrt{x}} \sum_{dm=k}\frac{\mu(d)}{m}\sum_{p\leq x, p\nmid N, k|d_p}p.
    \end{align*}
    Then we split the sum into two parts as in ~\cite{JW}.
    \begin{align*}
    S_1&=\sum_{k\leq y}\sum_{dm=k}\frac{\mu(d)}{m}\sum_{p\leq x, p\nmid N, k|d_p}p, \\
    S_2&=\sum_{y<k\leq 2\sqrt{x}}\sum_{dm=k}\frac{\mu(d)}{m}\sum_{p\leq x, p\nmid N, k|d_p}p.
    \end{align*}
    Here a variable $y$ is to be chosen later within $3\leq y\leq 2\sqrt{x}$. We treat $S_2$ using trivial estimate
    \begin{equation}
    \left|\sum_{dm=k}\frac{\mu(d)}{m}\right|\leq 1
    \end{equation}
    and Lemma 2.3, then we obtain
    \begin{equation}
    |S_2|\ll \sum_{y<k\leq 2\sqrt{x}} x\cdot \frac{x}{k^2} \ll \frac{x^2}{y}.
    \end{equation}
    Let $\pi_E(x;k)= \frac{\textrm{Li}(x)}{n_k}+E_k(x)$. Our goal for treating $S_1$ is making use of Lemma 2.6.
    First, we take care of the inner sum by partial summation
    \begin{align*}
    \sum_{p\leq x, p\nmid N, k|d_p}p&=\int_{2-}^x td\pi_E(t;k)\\
    &=x\pi_E(x;k)-\int_2^x\pi_E(t;k)dt\\
    &=\frac{x\textrm{Li}(x)}{n_k}-\int_2^x \frac{\textrm{Li}(t)}{n_k}dt+O\left(x|E_k(x)|+\int_2^x|E_k(t)|dt\right)\\
    &=\frac{1}{n_k}\textrm{Li}(x^2)+O\left(x\max_{t\leq x}|E_k(t)|+1\right).
    \end{align*}
    Then we deal with $S_1$ using the trivial estimate (10) and Lemma 2.1, we have
    \begin{equation}
    S_1=c_E\textrm{Li}(x^2)+O\left(x\max_{t\leq x}|E_2(t)|\right)+O\left(\frac{x^2}{y\log x}+\sum_{3\leq k\leq y} x\max_{t\leq x}|E_k(t)|+\sqrt{x}\right)
    \end{equation}
    where $$c_E=\sum_{k=1}^{\infty}\frac{1}{n_k}\sum_{dm=k}\frac{\mu(d)}{m}.$$

    Let $\widetilde{\pi_E}(x;k)=\#\{\ip: N(\ip)\leq x, \ip\nmid k\iif , \ip \textrm{ splits completely in }K(E[k])\}$.
    By Lemma 2.4, we have
    \begin{equation}
    \pi_E(x;k)=\frac{1}{2}\widetilde{\pi_E}(x;k)+O\left(\frac{x^{1/2}}{\log x}\right)+O(\log N) \textrm{ uniformly for }k\geq 3.
    \end{equation}
    For the detailed explanation, we refer to ~\cite[page 9]{AM}.
    By Lemma 2.5, we have
    \begin{equation}
    \widetilde{\pi_E}(x;k)-\frac{\textrm{Li}(x)}{[K(E[k]):K]}=\sum_{i=1}^{t(k)}\left(\pi_K(x,k\iif,\im_i)-\frac{\textrm{Li}(x)}{h(k\iif)}\right).
    \end{equation}
    Again using Lemma 2.5 to bound $t(m)$ and applying Lemma 2.6 as in ~\cite[page 10]{AM},
    \begin{equation}
    \sum_{3\leq k\leq \frac{x^{1/4}}{N(\iif)(\log x)^{C/2}}}\max_{t\leq x}\left|\widetilde{\pi_E}(t;k)-\frac{\textrm{Li}(t)}{[K(E[k]):K]}\right|\ll_{A,B}N\log{N}\frac{x}{(\log x)^{A+B+1}},
    \end{equation}
    where $C=C(A,B)$ is the corresponding positive constant in Lemma 2.6 for the positive constant $A+B+1$. \\
    Note that $T(\iq)\leq 6$. Writing $\widetilde{E_k}(x)=\widetilde{\pi_E}(x;k)-\frac{\textrm{Li}(x)}{[K(E[k]):K]}$, and using a bound for $\max_{t\leq x}|E_2(t)|$(See ~\cite[Lemma 2.3]{AM}), we have
    \begin{equation}
    S_1=c_E\textrm{Li}(x^2)+O_{A,B}\left(\frac{x^2}{(\log x)^B}\right)+O\left(\frac{x^2}{y\log x}+\sum_{3\leq k\leq y} x\max_{t\leq x}|\widetilde{E_k}(t)|+\frac{x^{3/2}y\log N}{\log x}\right)
    \end{equation}
    Now, taking $y=\frac{x^{1/4}}{N(\iif)(\log x)^{C/2}}$, we obtain
    \begin{equation}
    S_1=c_E\textrm{Li}(x^2)+O_{A,B}\left(\frac{x^2}{(\log x)^B}+x^{7/4}N(\iif)(\log x)^{C/2-1}+\frac{x^2 N\log N}{(\log x)^{A+B+1}}+\frac{x^{7/4}\log N}{N(\iif)(\log x)^{1+C/2}}\right).
    \end{equation}
    Note that $N=N(\iif)|d_K|$, where $d_K$ is the discriminant of $K$. Combining with estimate of $|S_2|$ in (12), it follows that
    \begin{equation}
    \sum_{p\leq x, p\nmid N}\frac{p}{d_p}=c_E\textrm{Li}(x^2)+O_{A,B}\left(\frac{x^2}{(\log x)^B}+\frac{x^2N\log N}{(\log x)^{A+B+1}}+x^{7/4}N(\log x)^C\right).
    \end{equation}
    Theorem 1.1 now follows. 

    \section{Proof of Theorem 1.2}
    Let $N$ be the conductor of a CM elliptic curve $E$ satisfying $N\leq (\log x)^A$. We use the following elementary identity
    $$
    k=\sum_{dm|k} m\mu(d)
    $$
    We unfold the sum similarly as in the proof of Theorem 1.1.
    \begin{align*}
    \sum_{p\leq x, p\nmid N} d_p &=\sum_{p\leq x, p\nmid N} \sum_{dm|d_p}m\mu(d) \\
    &=\sum_{k\leq 2\sqrt{x}}\sum_{dm=k}m\mu(d)\sum_{p\leq x, p\nmid N,k|d_p}1
    \end{align*}
    We introduce a variable $y$ and split the sum as shown in the proof of Theorem 1.1. The inequality in the last line is due to the primes $\ip$ in $K$ which have degree 2 over $\Q$ and split completely in $K(E[k])$.
    \begin{align*}
    \sum_{p\leq x, p\nmid N} d_p &=\pi_E(x;2)+\sum_{3\leq k\leq 2\sqrt{x}}\phi(k)\pi_E(x;k)\\
    &\leq\frac{2x}{\log x}+\sum_{3\leq k\leq y}\phi(k)\frac{1}{2}\widetilde{\pi_E}(x;k)+\sum_{y<k\leq 2\sqrt x}\phi(k)\pi_E(x;k).
    \end{align*}
    Let $S_1$, $S_2$ denote the second sum and the third sum respectively.
    \begin{align*}
    S_1&=\sum_{3\leq k\leq y}\phi(k)\frac{1}{2}\widetilde{\pi_E}(x;k),\\
    S_2&=\sum_{y<k\leq 2\sqrt x}\phi(k)\pi_E(x;k).
    \end{align*}
    Now, we use Lemma 2.5, and 2.7 to give an upper bound for each $\widetilde{\pi_E}(x;k)$.
    \begin{equation}
    \widetilde{\pi_E}(x;k)\leq 2\frac{t(k)x}{h(k\iif)\log\frac{x}{N(k\iif)}}\left\{1+O\left(\frac{\log\log 3\frac{x}{N(k\iif)}}{\log\frac{x}{N(k\iif)}}\right)\right\}
    \end{equation}
    Then we treat $S_1$ by (19), and $S_2$ by the trivial bound ($\pi_E(x;k)\ll \frac{x}{k^2}$) in Lemma 2.3. As a result, we obtain
    \begin{align*}
    S_1&\ll x\sum_{3\leq k\leq y}\frac{\phi(k)}{n_k\log\frac{x}{k^2N(\iif)}},\\
    S_2&\ll x\sum_{y< k\leq 2\sqrt{x}}\phi(k)\frac{1}{k^2}\ll x\log\frac{\sqrt{x}}{y},\\
    \end{align*}
    where the implied constants are absolute.
    Applying partial summation to $S_1$ with $\phi(k)^2\ll n_k$, and $\sum_{k\leq t}\frac{1}{\phi(k)}=A_1\log t+O(1)$,  we obtain 
    \begin{equation}
    S_1\ll x\log\log \frac{x}{N(\iif)} \ll_A x\log\log x,
    \end{equation}
    provided that $3\leq \frac{x}{y^2N(\iif)}$. 
    
    Choosing $y=\sqrt{\frac{x}{3N(\iif)}}$, it follows that
    \begin{equation}
    S_1+S_2\ll_A x\log\log x
    \end{equation}
    Therefore, Theorem 1.2 now follows.

    Note that the trivial bound in Theorem 1.2 given by Lemma 2.3 is $\ll x\log x$. The number field analogue of Brun-Titchmarsh Inequality(Lemma 2.7) contributed to the saving.

    \end{document}